\documentclass[11pt]{amsart}

\usepackage{amsmath, amsthm, amssymb}
\usepackage{graphicx, array}
\usepackage{longtable}
\usepackage{epstopdf}
\usepackage{mathrsfs}

\theoremstyle{definition}

\newtheorem{theorem}{Theorem}[section]
\newtheorem{proposition}{Proposition}
\newtheorem{lemma}{Lemma}

\newtheorem{remark}{Remark}

\title{The word problem for some classes of one relation Adian inverse semigroups}

\author{Muhammad Inam}
\email{muhammad.inam@aamu.edu}

\address{Department of Physics, Chemistry \& Mathematics\\
Alabama A\&M University\\
Huntsville, Alabama 35811 USA}

\begin{document}
\date{\today}

\begin{abstract} We show that if the Sch\"{u}tzenberger graph of every positive word, that contains an $R$-word only once as it's subword, is finite over an Adain presentation $\langle X|u=v\rangle$, then the Sch\"{u}tzenberger graph of every positive word is finite over the presentation $\langle X|u=v\rangle$. This enable us to solve the word problem ffor some classes of one relation Adian inverse semigroups.

\end{abstract}

\maketitle

\section{Introduction}

\bigskip

Throughout this paper $X$ is considered to be a non-empty set that denotes an alphabet. The set $R=\{ (u_i,v_i)|i\in I\}$, where $u_i,w_i\in X^+$, is called the set of \textit{positive relations}. The pair $\langle X|R\rangle$ is called a \textit{positive presentation}. The  semigroup generated by the set $X$ and having a set $R$ of relations, is denoted by $Sg\langle X|R\rangle$ and the group generated by the set $X$ and having a set $R$ of relations is denoted by $Gp\langle X|R\rangle$. There exists a natural homomorphism $\phi$:$Sg\langle X|R\rangle$ $\to$ $Gp\langle X|R\rangle$. 

We can construct two undirected graphs corresponding to a positive presentation. The \textit{left graph} of the presentation $\langle X|R\rangle$ is  denoted by $LG\langle X|R\rangle$. The vertices of $LG\langle X|R\rangle$ are labeled by the elements of $X$, and there is an edge corresponding to every relation $(u_i,v_i)\in R$, that connects the prefix letters of $u_i$ and $v_i$ together. Similarly, the \textit{right graph} of the presentation $\langle X|R\rangle$ is denoted by $RG\langle X|R\rangle$, and it can be obtained by connecting the suffix letters of $u_i$ and $v_i$ together, for every $(u_i,v_i)\in R$.  A closed path in  $LG\langle X|R\rangle$ is called a \textit{left cycle} and a closed path in $RG\langle X|R\rangle$ is called a \textit{right cycle}. If for some presentation $\langle X|R\rangle$, there is no closed path (cycle) in $LG\langle X|R\rangle$ and $RG\langle X|R\rangle$, then the presentation is called a \textit{cycle free} presentation. The cycle free presentations are also called the Adian presentations,  because these presentations were first studied by S. I. Adian \cite{Adian}. 

A semigroup $S$ is called an \textit{inverse semigroup} if for every element $a\in S$ there exists a unique element $b\in S$ such that $aba=a$ and $bab=b$.  The Unique element $b$ is denoted by $a^{-1}$. The idempotents commute in an inverse semigrpoup, and the product of two idempotents is an idempotent. The \textit{natural partial order} on an inverse semigroup $S$ is defined as $a\leq b$ if and only if $aa^{-1}b=a$, for some $a,b\in S$.  A congruence relation $\sigma$ is defined on $S$, for $a,b\in S$, by $a\sigma b$ if and only if there exists an element $c\in S$ such that $c\leq a,b$. It turns out that $\sigma$ is the minimum group congruence relation on $S$, i.e., $S/\sigma$ is the maximum group homomorphic image of $S$. An inverse semigroups can also be presented by a set of generators and a set of relations, just like groups and semigroups. If $S=Inv\langle X|R\rangle$, then $S/\sigma$ is the group $Gp\langle X|R\rangle$. The set of idempotents of $S$ is denoted by $E(S)=\{e\in S: e^2=e\}$.  An inverse seimigroup is called \textit{E-unitary} if $\sigma^{-1}=E(S)$. All these facts about the inverse semigroups along with more details can be found in the text \cite{Lawson}.

 J. B. Stephen \cite{ST} introduced the notion of \textit{Sch\"{u}tzenberger graphs} to solve the word problem for inverse semigroups. If $M=Inv\langle X|R\rangle$ is an inverse semigroup then we may consider the corresponding Cayley graph $\Gamma(M,X)$. The vertices of this graph are labeled by the elements of $M$ and there exists a directed edge labeled by $x\in X\cup X^{-1}$ from the vertex labeled by $m_1$ to the vertex labeled by $m_2$ if $m_2 = m_1x$.   The Cayley graph $\Gamma(M,X)$ is not necessarily strongly connected, unless $M$ happens to be a group, because it may happen that when there is an edge labeled by $x$ from $m_1$ to $m_2$ there is not an edge labeled by $x^{-1}$ from $m_2$ to $m_1$ (so, $m_2 = m_1x$, but $m_1 \neq m_2x^{-1}$). The strongly connected components of $\Gamma(M,X)$ are called the \textit{Sch\"{u}tzenberger graphs} of $M$. For any word $u\in (X\cup X^{-1})^*$ the strongly connected component of $\Gamma(M,X)$ that contains the vertex labeled by $u$ is the \textit{Sch\"{u}tzenberger graph of $u$} and it is denoted by $S\Gamma(M,X,u)$. In \cite{ST} it is shown that the vertices of $S\Gamma(M,X,u)$ are precisely those vertices that are labeled by the elements of the $\mathscr{R}$-class of $u$, i.e., $R_u = \{m\in M \mid mm^{-1} = uu^{-1}\}$.


 For any word $u\in (X\cup X^{-1})^*$, it is useful to consider the \textit{Sch\"{u}tzenberger automaton} $(uu^{-1},S\Gamma(M,X,u),u)$ with initial vertex $uu^{-1}\in M$, terminal vertex $u\in M$ and with the Sch\"{u}tzenberger graph of $u$ as the underlying graph of the Sch\"{u}tzenberger automaton of $u$. The language accepted by this automaton is a subset of $(X\cup X^{-1})^*$ and will be denoted as $L(u)$.
 \[ L(u) = \{w\in (X\cup X^{-1})^* \mid w \mbox{ labels a path  from } uu^{-1} \mbox{ to } u \mbox{ in } S\Gamma(M,X,u)\}. \]
Here, $u$ and $w$ may be regarded both as elements of $(X\cup X^{-1})^*$ and as elements of $M$. Thus, $L(u)$ may be regarded as a subset of $(X\cup X^{-1})^*$ or as a subset of $M$.

The following result of Stephen \cite{ST} plays a key role in solving the word problem for inverse semigroups.
\begin{theorem}\label{Stephen's thm}
Let $M=Inv\langle X|R\rangle$ and let $u,v\in (X\cup X^{-1})^*$. 
\begin{enumerate}
\item $L(u)= \{w \mid w \geq u$ in the natural partial order on $M\}$.
\item The following are equivalent:
\begin{enumerate}
\item[(i)] $u = v$ in $M$.
\item[(ii)] $L(u)=L(v)$.
\item[(iii)] $u\in L(v)$ and $v\in L(u)$.
\item[(iv)] $(uu^{-1},S\Gamma(M,X,u), u)$ and $(vv^{-1},S\Gamma(M,X,v),v)$ are isomorphic as automata.
\end{enumerate}
\end{enumerate}

\end{theorem}

 We briefly describe the iterative procedure described by Stephen \cite{ST} for building a Sch\"{u}tzenberger graph. Let $Inv\langle X|R\rangle$  be a presentation of an inverse monoid.

   Given a word $u=a_1a_2...a_n\in (X\cup X^{-1})^*$, the \textit{linear graph} of $u$  is the birooted inverse word graph $(\alpha_u,\Gamma_u,\beta_u)$ consisting of the set of vertices

   \begin{center}

   $V((\alpha_u,\Gamma_u,\beta_u))=\{\alpha_u,\beta_u,\gamma_1,...,\gamma_{n-1}\}$

   \end{center}

   and edges

   \begin{center}

   $(\alpha _u,a_1, \gamma _1),(\gamma _1,a_2,\gamma _2),..., (\gamma _{n-2},a_{n-1},\gamma _{n-1}),(\gamma _{n-1},a_n,\beta _u)$,

   \end{center}

    together with the corresponding inverse edges.

    Let $(\alpha , \Gamma ,\beta )$ be a birooted inverse word graph over $X\cup X^{-1}$. The following operations may be used to obtain a new birooted inverse word graph $(\alpha ',\Gamma ',\beta ')$:

    $\bullet$ \textbf{Determination} or \textbf{folding:} Let $(\alpha,\Gamma,\beta)$ be a birooted inverse word graph with vertices $v,v_1,v_2$, with $v_1\neq v_2$, and edges $(v,x,v_1)$ and $(v,x,v_2)$ for some $x\in X\cup X^{-1}$.

    Then we obtain a new birooted inverse word graph $(\alpha',\Gamma',\beta')$ by taking the quotient of $(\alpha,\Gamma,\beta)$ by the equivalence relation that identifies the vertices $v_1$ and $v_2$ and the two edges. In other words, edges with the same label coming out of a vertex are folded together to become one edge.

    $\bullet$ \textbf{Elementary $\mathscr{P}$-expansion:} Let $r=s$ be a relation in $R$ and suppose that $r$ can be read from $v_1$ to $v_2$ in $\Gamma$, but $s$ cannot be read from $v_1$ to $v_2$ in $\Gamma$. Then we define $(\alpha',\Gamma',\beta')$ to be the quotient of $\Gamma \cup (\alpha _s,\Gamma_s,\beta_s)$ by the equivalence relation that identifies vertices $v_1$ and $\alpha_s$ and vertices $v_2$ and $\beta_s$. In other words. we  ``sew" on a linear graph for $s$ from $v_1$ to $v_2$ to complete the other half of the relation $r=s$.

    An inverse word graph is \textit{deterministic} if no folding can be performed and \textit{closed} if it is deterministic and no elementary expansion can be performed over a presentation $\langle X|R\rangle$. Note that given a finite inverse word graph it is always possible to produce a determinized form of the graph, because determination reduces the number of vertices. So, the process of determination must stop after finitely many steps, We note also that the process of folding is confluent \cite{SG} .

    If $(\alpha_1,\Gamma_1, \beta_1)$ is obtained from $(\alpha,\Gamma,\beta)$ by an elementary $\mathscr{P}$-expansion, and $(\alpha_2,\Gamma_2,\beta_2)$ is the determinized  form of $(\alpha_1,\Gamma_1,\beta_1)$, then we write $(\alpha,\Gamma,\beta)$\\$\Rightarrow (\alpha_2,\Gamma_2,\beta_2)$ and say that $(\alpha_2,\Gamma_2,\beta_2)$ is obtained from $(\alpha, \Gamma,\beta)$ by a  \textit{$\mathscr{P}$-expansion}. The reflexive and transitive closure of $\Rightarrow$ is denoted by $\Rightarrow ^*$.

    For $u\in (X\cup X^{-1})^*$, an \textit{approximate graph} of $(uu^{-1}, S\Gamma(u), u)$ is a birooted inverse word graph $A=(\alpha,\Gamma,\beta)$ such that $u\in L[A]$ and $y\geq u$ holds in $M$ for all $y\in L[A]$. Stephen showed in \cite{ST} that the linear graph of $u$ is an approximate graph of $(uu^{-1}, S\Gamma(u), u)$. He also proved the following:

    \begin{theorem}\label{closure}
    Let $u\in (X\cup X^{-1})^*$ and let $(\alpha,\Gamma,\beta)$ be an approximate graph of $(uu^{-1},S\Gamma(u), u)$. If $(\alpha,\Gamma,\beta)\Rightarrow^*(\alpha',\Gamma',\beta')$ and $(\alpha',\Gamma',\beta')$ is closed, then $(\alpha',\Gamma',\beta')$ is the Sch\"{u}tzenberger automaton for $u$.
    \end{theorem}

    In \cite{ST}, Stephen showed that the class of all birooted inverse words graphs over $X\cup X^{-1}$ is a co-complete category  and that the directed system of all finite $\mathscr{P}$-expansions of a linear graph of $u$ has a direct limit. Since the directed system includes all possible $\mathscr{P}$-expansions, this limit must be closed. Therefore, by \ref{closure}, the Sch\"{u}tzenberger graph of $u$ is the direct limit.

    \textbf{Full $\mathscr{P}$- expansion (a generalization of the concept of $\mathscr{P}$-\\ expansion):} A full $\mathscr{P}$-expansion of a birooted inverse word graph $(\alpha,\Gamma,\beta)$ is obtained in the following way:

    $\bullet$ Form the graph $(\alpha',\Gamma',\beta')$, which is obtained from $(\alpha,\Gamma,\beta)$ by performing all possible elementary $\mathscr{P}$-expansions of $(\alpha,\Gamma,\beta)$, relative to $(\alpha,\Gamma,\beta)$. We emphasize  that an elementary $\mathscr{P}$-expansion may introduce a path labeled by one side of relation in $R$, but we do not perform an elementary $\mathscr{P}$-expansion that could not be done to $(\alpha,\Gamma,\beta)$ when we do a full $\mathscr{P}$-expansion.

    $\bullet$ Find the determinized form $(\alpha_1,\Gamma_1,\beta_1)$, of $(\alpha',\Gamma',\beta')$.

    The birooted inverse word graph $(\alpha_1,\Gamma_1,\beta_1)$ is called the full $\mathscr{P}$-expansion of $(\alpha,\Gamma,\beta)$. We denote this relationship by $(\alpha,\Gamma,\beta)\Rightarrow_f (\alpha_1,\Gamma_1,\beta_1)$.  If $(\alpha_n,\Gamma_n,\beta_n)$ is obtained from $(\alpha,\Gamma ,\beta)$ by a sequence of full $\mathscr{P}$-expansions then we denote this by   $(\alpha,\Gamma ,\beta)\Rightarrow^*_f(\alpha_n,\Gamma_n,\beta_n)$.





\section{The Word Problem For Adian Inverse Semigroups}

Adian conjectured \cite{Adian} that the word problem is decidable for Adain semigroups. The following result was first proved by Adian \cite{Adian} for only finite Adian presentations. Latter Remmers \cite{Remmers} proved the same result, by using geometric techniques, for any Adian presentation. 

\begin{theorem} An Adain semigroup embeds in the Adian group with the same presentation. 

\end{theorem}

Magnus \cite{Lyndon} proved that the word problem is decidable for one relator groups. Since Adian semigroups embed in Adian groups, therefore, it follows that the word problem is decidable for one relation Adian semigroups as well. But, the question, whether the word problem is decidable for one relation Adain inverse semigroups or not, has not been answered yet. 

The following theorem is proved in \cite{E-unitary}
\begin{theorem}\label{E-unitary} Adian inverse semigroups are $E$-unitary.
\end{theorem}

By using Theorem \ref{E-unitary}, we can make the following remark about one relation Adian inverse semigroups. 

\begin{remark} Let $M=Inv\langle X|u=v\rangle$ be an Adian inverse semigroup. Then the membership problem for the set of idempotents of $M$, $E(M)$, is decidable.
\end{remark}

The following result has been proved in \cite{AD}

\begin{theorem}  Let $Inv\langle X|R\rangle$ be a finitely presented Adian inverse semigrooup. Then the Sch\"{u}tzenberger graph of $w$, for all $w\in (X\cup X^{-1})^+$, is finite if and only if the Sch\"{u}tzenberger graph of $w'$ is finite, for all $w'\in X^+$. 
\end{theorem} 

In this paper, we prove the following result, which enable us to decide the word problem for some classes one relation Adian inverse semigroups.

\begin{theorem} Let $\langle X|u=v\rangle$ be an Adian presentation.  The Sch\"{u}tzenberger graph of every positive word, that contains an $R$-word only once as it's subword, is finite over the presentation $\langle X|u=v\rangle$ if and only if the Sch\"{u}tzenberger graph of every positive word is finite over the presentation $\langle X|u=v\rangle$. 
\end{theorem}

\section{Main Theorem}
The following lemma has been proved in \cite{AD}. 

\begin{lemma}\label{1} Let $M=Inv\langle X,R\rangle$ be an Adian inverse semigroup and $w\in X^+$. Then no two edges fold together in Stephen's process of constructing approximations of the Sch\"{u}tzenberger graph of $w$. 
\end{lemma}

So, no folding of edges occurs at any step of the proof of the following lemma.

\begin{theorem} Let $\langle X|u=v\rangle$ be an Adian presentation. The Sch\"{u}tzenberger graph of every positive word, that contains an $R$-word only once as it's subword, is finite over the presentation $\langle X|u=v\rangle$ if and only if the Sch\"{u}tzenberger graph of every positive word is finite over the presentation $\langle X|u=v\rangle$. 
\end{theorem}

\begin{proof} We just prove the direct statement. The converse is obvious. Let $w$ be a given positive word. If $w$ doesn't contain any $R$-word as it's subword then our claim follows immediately. If $w$ contains only one $R$-word as it's subword, such that the $R$-word occurs only once in $w$ then the finiteness of $S\Gamma(w)$ follows by the hypothesis of the theorem.  

We assume that the statement of the theorem is true for any positive word that contains less than $n$ number of (not necessarily distinct) $R$-words as it's subword, where $n\in \mathbb{N}$ and $n\geq 2$. We assume that $w$ contains $n$ number of (not necessarily distinct) $R$-words as it's subwords and we show that $S\Gamma(w)$ is finite. 



 We factorize $w$ as $w\equiv w_1w_2$ such that both $w_1$ and $w_2$ contain less than $n$ number of $R$-words as their subwords. The Sch\"{u}tzenberger graphs of $w_1$ and $w_2$ are finite by the induction hypothesis. We construct the linear automaton $(\alpha_0, \Gamma_0(w),\beta_0)$. We attach $S\Gamma(w_1)$ and $S\Gamma(w_2)$ to $(\alpha_0, \Gamma_0(w),\beta_0)$ at the segments labeled by $w_1$ and $w_2$, respectively.  We denote the resulting graph by $S_0$.
 
 \begin{figure}[h!]
\centering
\includegraphics[trim = 0mm 0mm 0mm 0mm, clip,width=2.8in]{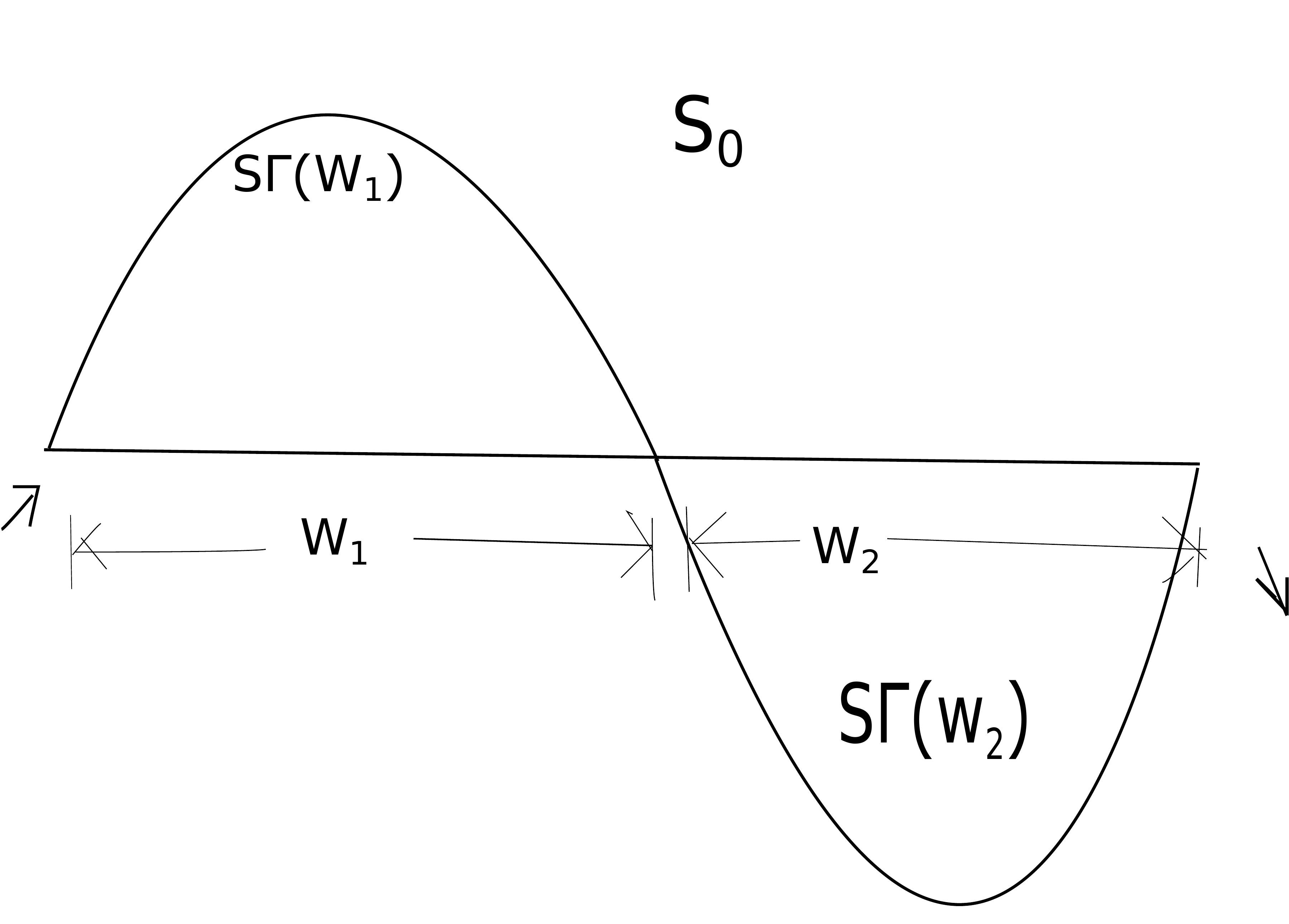}
\caption{ }
\label{fig301}
\end{figure}

 Every transversal (a positively labeled path from the initial vertex $\alpha_0$ to the terminal vertex $\beta_0$) of $S_0$ passes through the vertex $\gamma$, where $\gamma$ is the terminal vertex of the segment labeled by $w_1$ and the initial vertex of segment labeled by $w_2$, in $S_0$. We know that $S\Gamma (w_1)$ and $S\Gamma (w_2)$ are complete in the sense that we can not find any unsaturated segment labeled by an $R$-word in these graphs. So, the only place where we can find an unsaturated segment labeled by an $R$-word, is at $\gamma$, that is, we can find an unsaturated segment labeled by an $R$-word that  passes through the vertex $\gamma$ along a transversal of $S_0$.
 
There are only finite number of transversals in $S_0$, as $S_0$ is a finite graph. We consider those transversals of $S_0$ that contain an unsaturated segment labeled by an $R$-word. We assume that these transversals are labeled by $t_1,t_2,...,t_m$, for some $m\in \mathbb{N}$. 

 \begin{figure}[h!]
\centering
\includegraphics[trim = 0mm 0mm 0mm 0mm, clip,width=2.8in]{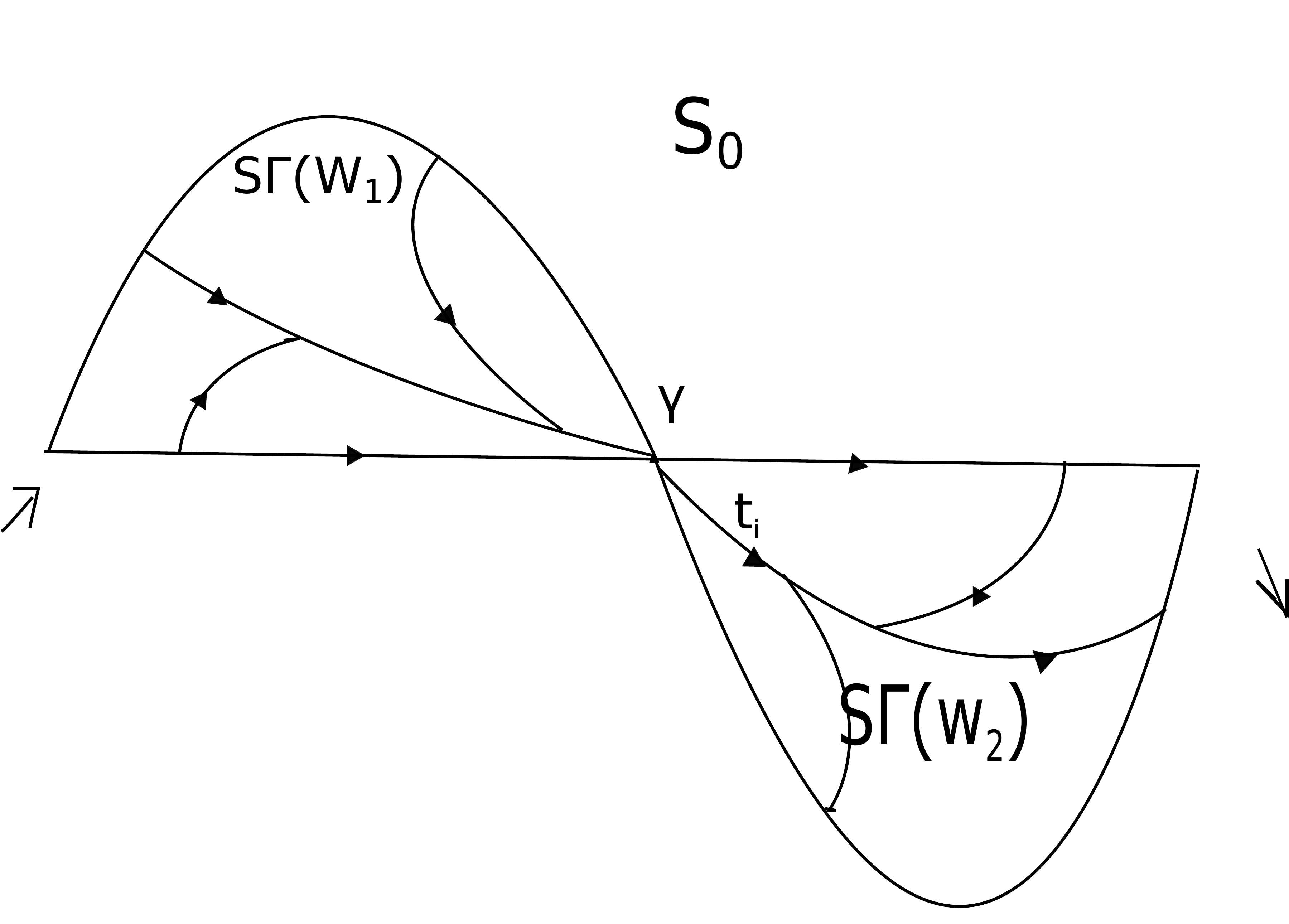}
\caption{ }
\label{fig301}
\end{figure}

Nest, we show that $S\Gamma(t_i)$ is finite, for all $i\in \{1,2,...,m\}$.

\textbf{\textit{Fact: 1}} \textit{Let $\langle X|u=v\rangle$ be an Adian presentation. If none of the $R$-word is a subword of the other $R$-word, and no $R$-word overlaps with itself or with the other $R$-word, then $S\Gamma(t)$ for all $t\in X^+$, is finite over the presentation $\langle X|u=v\rangle$. }

\begin{proof}
\textit{Proof of the fact:} Let $t$ be a positive word. If $t$ does not contain any $R$-word as it's subword then there is nothing to prove, because $S\Gamma(t)$ is finite in this case.

We assume that $t$ contains $n$ number of (not necessarily distinct) $R$-words, for some $n \in \mathbb{N}$. We construct the linear automaton of $t$, $(\alpha_0,\Gamma_0(t), \beta_0)$.  We sew on a path labeled by the other side of the relation from the initial vertex to the terminal vertex of every segment, labeled by an $R$-word, of $(\alpha_0,\Gamma_0(t), \beta_0)$. As we know that none of the edges get identified with each other because of Lemma \ref{1}, we denote the resulting automaton by $(\alpha_1,\Gamma_1(t), \beta_1)$. Since none of the $R$-word is subword of the other $R$-word, therefore, we cannot find an unsaturated sub-segment labeled by an $R$-word  of the newly attached segments of $(\alpha_1,\Gamma_1(t), \beta_1)$. We know from Proposition 3(iii) of \cite{AD}, that $(\alpha_1,\Gamma_1(t), \beta_1)$ is a directional graph in the sense that every vertex of $(\alpha_1,\Gamma_1(t), \beta_1)$ lies on a transversal from $\alpha_1$ to $\beta_1$. Since none of the $R$-word overlaps with itself or with the other $R$-word, therefore, we cannot find an unsaturated segment labeled by an $R$-word that either starts from a vertex that lies before the initial vertex of a segment that was attached in the previous step and terminating at an interior vertex of that segment, or that starts from an interior vertex of a segment that was attached in the previous step and terminates at a vertex that lies after the terminal vertex of that segment. So, we cannot expand $(\alpha_1,\Gamma_1(t), \beta_1)$ further by applying Full $P$-expansion.  Hence, the underlying graph of $(\alpha_1,\Gamma_1(t), \beta_1)$ is $S\Gamma(t)$, and it's finite.

\end{proof}

We assume that $S\Gamma (t_i)$, for some $i\in \{1,2,...,m\}$, is infinite,  that means, whenever we sew on a path labeled by an $R$-word from the initial vertex to the terminal vertex of an unsaturated segment labeled by an $R$-word,  it creates another unsaturated segment labeled by an $R$-word.  It follows from the above fact that, this is only possible when at least one of the following conditions hold.   
\begin{enumerate}
\item $v$ is a subword of $u$. 
 
 \item One of the $R$-word or both of the $R$-words overlap with themselves and/or both of the $R$-word overlap with each other. 
 \end{enumerate}
 If one of the $R$-word is a subword of the other, then Sch\"{u}tzwnberger  graphs of both the $R$-words are infinite. This contradicts the fact that every positive word that contain an $R$-word only once as it's subword has finite Sch\"{u}tzwnberger graph.   

The other possibility is this, that the $R$-words, $u$ and $v$ overlap with themselves and/or with each other.  Here it is important to remind the readers that the Sch\"{u}tzenberger graph of a positive word over an Adian presentation is a directional graph in the sense that each of it's vertex lies on a transversal from the initial vertex of the graph (called the source)  to the terminal vertex of the graph (called the sink), by proposition 3(iii) of \cite{AD}. n this case, when we sew on a path labeled by an $R$-word, it creates another unsaturated segment labeled by an $R$-word, that either starts from a vertex of a transversal that lies before the initial vertex of the newly attached segment and terminates at an interior vertex the newly attached segment or it starts from an interior vertex of the newly attached segment and terminates at a vertex of a transversal that occurs after the terminal vertex of the newly attached segment.  It is possible that the transversal labeled by $t_i$ contain more than one unsaturated segments labeled by the $R$-words passing through the vertex $\gamma$. No two edges fold together in the construction of $S\Gamma(t_i)$ by Lemma \ref{1}. So, the transversal labeled by $t_i$ contains at least one unsaturated segment labeled by an $R$-word passing through $\gamma$ that makes $S\Gamma(t_i)$ to be infinite, otherwise, the word $t_i$ can be factored as $t_i=t_{i_{1}} t_{i_{2}}$, where $t_{i_{1}}$ labels the path from $\alpha_0$ to the vertex $\gamma$ and $t_{i_{2}}$ labels the path from $\gamma$ to the vertex $\beta_0$. Note that, $S\Gamma (t_{i{1}})$  embeds in $S\Gamma (w_1)$ and $S\Gamma (t_{i_{2}})$ embeds in $S\Gamma (w_2)$, therefore $S\Gamma (t_{i_{1}})$  and $S\Gamma (t_{i_{2}})$ are finite. If $t_i$ contains no $R$-word passing through $\gamma$ that is making $S\Gamma(t_i)$ to be infinite, then $S\Gamma(t_i)$ will be a finite graph, because $S\Gamma(t_i)$ can be constructed by by attaching $S\Gamma(t_{i_1})$ and   $S\Gamma(t_{i_2})$ to the linear automaton of $t_i$ at the segments labeled by $t_{i_1}$ and $t_{i_2}$, and attaching the maximal finite Sch\"{u}tzenberger graph(s) generated by a subword of $t_i$, that contain(s) the $R$-word that labels an the unsaturated segment passing through $\gamma$. This is a contradiction to our hypothesis that $S\Gamma(t_i)$ is infinite. 


There are four cases of overlapping of an $R$-word with itself and/or with the other $R$-word. 

\begin{enumerate}

  \item $u\equiv xsx$ or/and $v\equiv yty$ for some $x,y\in X^+$ and $s,t\in X^*$, where $x$ and $y$ are the maximal prefixes of $u$ and $v$ that are also the suffixes of $u$ and $v$, and there is no overlap between $u$ and $v$. 
 
 \item $u\equiv xy$ and $v\equiv yz$ or $v\equiv zx$, for some $x,y,z\in X^+$. If $v\equiv yz$ then $y$ is the maximal suffix of $u$ that is a prefix of $v$ and no prefix of $u$ is a suffix of $v$. If $v\equiv zx$ then $x$ is the maximal prefix of $u$ that is a suffix of $v$ and no suffix of $u$ is a prefix of $v$. 
 
 \item  $u\equiv xsx$ or/and $v\equiv yty$ for some $x,y\in X^+$ and $s,t\in X^*$, where $x$ and $y$ are the maximal prefixes of $u$ and $v$ that are also the suffixes of $u$ and $v$, and  either a prefix of $u$ is a suffix of $v$ or a suffix of $u$ is a prefix of $v$.
 
\item A prefix of $u$ is a suffix of $v$ and a suffix of $u$ is prefix of $v$.  
 
 \end{enumerate}

 \textit{Csse 1:}  First, we assume that $u\equiv xsx$ for some $x\in X^+$ and $s\in X^*$ (i.e., $u$ has same prefix and suffix), whereas $v$ is any positive word such that no proper prefix of $v$ is a suffix of $v$ and $v$ does not overlap with $u$ either. We show that  in this case the Sch\"{u}tzenberger graph of $t_i$ cannot be  infinite because of an unsaturated segment labeled by an $R$-word passing through the vertex $\gamma$.

 If the unsaturated segment passing through the vertex $\gamma$ is labeled  by $u$ then we sew on a path labeled by $v$, from the initial vertex to the terminal vertex of the segment labeled by $u$. Since $v$ does not overlap with itself or with $u$ , therefore we can't find an unsaturated segment labeled by an $R$-word, where we can sew on a new path labeled by an $R$-word. So, $S\Gamma(t_i)$ remains finite.

 If the unsaturated segment passing through the vertex $\gamma$ is labeled by $v$, then we sew on a path labeled by $u$ from the initial vertex to the terminal vertex of the segment labeled by $v$, that creates a first generation bounded region. Since $u$ overlaps with itself, therefore if the segment labeled by $v$ is followed by a segment labeled by $xs$ or a segment labeled by $sx$ is followed by the segment labeled by $v$, in the transversal labeled by $t_i$, then we can find some unsaturated segment labeled by $u$ that either starts from a vertex of the transversal $t_i$ that lies before the initial vertex of the segment labeled by $v$ and terminates at an interior vertex of the segment labeled by $u$ or it starts from an interior vertex of the segment labeled by $u$ and terminates at a vertex of the transversal $t_i$ that lies after the terminal vertex of the segment labeled by $v$.  We attach a new segment labeled by $v$ from the initial vertex to the terminal vertex of the unsaturated segment(s) to obtain some  second generation bounded regions.  Otherwise the process of obtaining higher generation bounded regions starting from an unsaturated labeled by an $R$-word  passing through $\gamma$ terminates at the previous step.  The $R$-word $v$ does't overlap with itself or with $u$, therefore we cannot find an unsaturated segment labeled by an $R$-word, that either terminates at an interior vertex of the segment(s) attached in the previous step or that starts from an interior vertex of the segment(s) attached in the previous step,  where we can attach a new segment labeled by the other side of the relation.  So,  $S\Gamma(t_i)$ cannot be  infinite.

 Now we assume that $u\equiv xsx$ and $v\equiv yty$ for some $x,y\in X^+$, $s,t\in X^*$, and $u$ and $v$ do not overlap with each other. Note that $x$ and $y$ cannot be the same words otherwise the presentation $\langle X|u=v\rangle$ will not be an Adian presentation. However, $s$ and $t$ can be the same words. We show that in both the cases, when $s\equiv t$ and when $s\not\equiv t$, $S\Gamma(t_i)$ cannot be  infinite because of an unsaturated segment labeled by an$R$-word passing through $\gamma$.

 We first assume that $s\not\equiv t$.  Without loss of generality we can assume that the unsaturated segment passing through the vertex $\gamma$ is labeled by $v$  A similar argument can be used when the passing through $\gamma$ is labeled by $u$. We consider the linear automaton of $t_i$ and we sew on a segment labeled by $u$ from the initial vertex to the terminal vertex of the segment labeled by $v$ and obtain a first generation bounded region.  If the word $v$ is followed by the word $xs$ or the word $sx$ is followed by $v$ in $t_i$, then we can find some unsaturated segments labeled by $u$. We sew on  segments labeled by $v$ from the initial vertex to the terminal vertex of the unsaturated segments labeled by $u$ created in the previous step. If the word $xsv$ is followed by the word $yt$ or the the word $ty$ is followed by $vsx$ in $t_i$, then we can find some unsaturated segments labeled by $v$.  We sew on  segments labeled by $u$ from the initial vertex to the terminal vertex of the unsaturated segments labeled by $v$ created in the previous step.  We observe that every time we apply the  $P$-expansion, we utilize  some of the segments of the linear automaton of $t_i$.  The process of application of elementary expansion will eventually  terminate, as $t_i$ is a word of finite length. So, the unsaturated segment labeled by an $R$-word passing through the vertex $\gamma$ cannot make $S\Gamma(t_i)$ to be an infinite graph.

Now we assume that $s\equiv t$ and the unsaturated segment passing through the vertex $\gamma$ is labeled by $v$. A similar argument can be used when the unsaturated segment passing through $\gamma$ is labeled by $u$. We consider the linear automaton of $t_i$ and we sew on a segment labeled by $u$ from the initial vertex to the terminal vertex of the segment labeled by $v$ and obtain a first generation bounded region.  If the word $v$ is followed by the word $xs$ or the word $sx$ is followed by $v$ in $t_i$, then we can find some unsaturated segments labeled by $u$. We sew on  segments labeled by $v$ from the initial vertex to the terminal vertex of the unsaturated segments labeled by $u$ created in the previous step and obtain some second generation bounded regions. If  the word $xsvsx$ is a subword of  $t_i$, where $\gamma$ is an interior vertex of the segment labeled by $v$, then we will have second generation bounded regions on both the sides of the first generation bounded region created in the previous step.  Since $s\equiv t$, therefore we can find an unsaturated segment labeled by $v$ starting from the segment labelled by $v$, sewed to obtain the second generation bounded region, and terminating at an interior vertex of the other segment labeled by $v$,  sewed to obtain the second generation bounded region, as shown in the following figure.

\begin{figure}[h!]
\centering
\includegraphics[trim = 0mm 0mm 0mm 0mm, clip,width=2.8in]{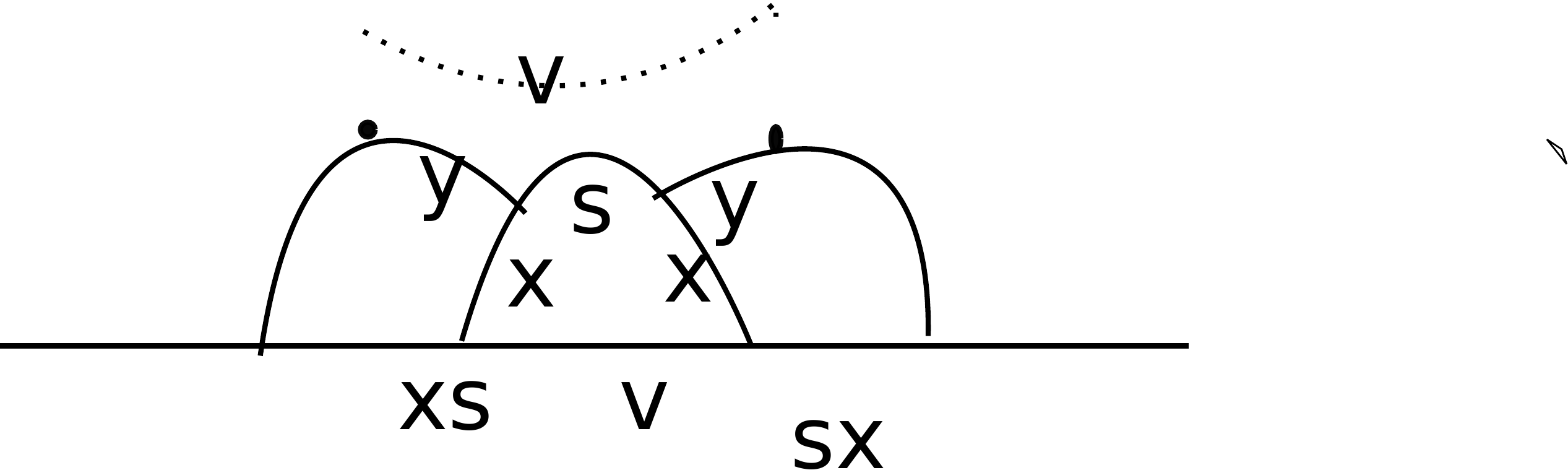}
\caption{ }
\label{fig301}
\end{figure}

If the word $xsv$ is followed by $yt$, the word $ty$ is followed by $vsx$, or the word $ytxsvsxty$ is a subword of $t_i$ then we find some unsaturated segment(s) labeled by $v$. We sew on the segment(s) labeled by $u$ from the initial vertex to the terminal vertex  of the unsaturated segment(s) labeled by $v$ and obtain the third generation bounded regions. We continue this  process of applying $P$-expansion and observe that the subgraph of $S\Gamma(t_i)$ which is generated by the segment labeled by $v$ passes though $\gamma$, grows like a pyramid, that is, the first layer of bounded regions along the transversal labeled by $t_i$ contain maximum number of bounded regions, the second layer of bounded regions, those bounded regions whose one side lies on the boundary of the first layer bounded regions,  contain fewer bounded regions than the first layer of bounded regions. Similarly, the third layer of bounded regions contain  fewer than second layer of bounded regions.  So, if the base of this pyramid contain $N$ number of bounded regions then the second layer contains $(N-2)$ bounded regions, the third layer contains $(N-4)$ bounded regions and son and so forth.  Since, $t_i$ is a word of finite length, therefore, there can be only finite number of bounded regions along the transversal labeled by $t_i$.  Hence, the segment labeled by $v$ passing through $\gamma$ cannot make $S\Gamma(t_i)$ to be an infinite graph.

\textit{Case 2:} Without loss of generality, we assume that $u\equiv xy$ and $v\equiv yz$, where $y$ is the maximal suffix of $u$ that is a prefix of $v$, no prefix of $u$ is a suffix of $v$ and none of the $R$-words have same prefix and suffix. We also assume that the unsaturated segment passing through $\gamma$ is labeled by $v$. 

We construct the linear automaton of $t_i$., and we sew on a segment labeled by $u$ from the initial vertex to the terminal vertex of the segment labeled by $v$ and obtain a first generation bounded region. If the word $z$ is followed by $v$ in $t_i$, then we can find an unsaturated segment labeled by $v$ starting from a vertex of the transversal $t_i$ and terminating at an interior vertex of the segment labeled by $u$, attached in the previous step. We sew on a segment labeled by $u$ from the initial vertex to the terminal vertex of the unsaturated segment labeled by $v$, and obtain a second generation bounded region. If the word $z$ is followed by the word $zv$, where $v$ is the segment passing through $\gamma$, in $t_i$, then once again we can find an unsaturated segment labeled by $v$ starting from a vertex of $t_i$ and terminating at an interior vertex of the segment labeled by $u$, attached in the previous step. So, we sew on a segment labeled by $u$ from the initial vertex to the terminal vertex of the unsaturated segment labeled by $v$, and obtain a third generation bounded region. We continue this process of constructing higher and higher generation bounded regions and observe that every bounded region shares a boundary edge with the transversal $t_i$. So, this construction process eventually terminates, as $t_i$ is a word of finite length. Hence the unsaturated segment of $t_i$ labeled by $v$ passing through $\gamma$ can't make $S\Gamma(t_i)$ to be an infinite graph.

\textit{Case 3:} The relation $(u,v)$ can satisfy case $1$ and case $2$ simultaneously.  So, without loss of generality, we assume that a suffix of $u$ is a prefix of $v$, that is, $u\equiv x_1y_1$ and $v\equiv y_1z_1$ for some $x_1, y_1,z_1 \in X^+$. We consider all of the possibilities of case $1$, and shaw that $S\Gamma(t_i)$ remains finite in each case. 

First, we assume that $u\equiv xsx$ for some $x\in X^+$ and $s\in X^*$ (i.e., $u$ has same prefix and suffix), whereas $v$ is any positive word such that no proper prefix of $v$ is a suffix of $v$ and $v$ overlaps with $u$ as mentioned in the previous paragraph. We show that  in this case the Sch\"{u}tzenberger graph of $t_i$ cannot be  infinite because of an unsaturated segment labeled by an $R$-word passing through the vertex $\gamma$.  
 
 
  If the unsaturated segment passing through the vertex $\gamma$ is labeled by $u$, then we sew on a path labeled by $v$ from the initial vertex to the terminal vertex of the segment labeled by $u$, that creates a first generation bounded region. Since $v$ doesn't overlap with itself, therefore if the segment labeled by $u$ is followed by a segment labeled by $x_1$, in the transversal $t_i$, then we can find an unsaturated segment labeled by $u$ that starts from a vertex of $t_i$ and terminates at an interior vertex of the segment labeled by $v$, that was attached in the previous step.  We attach a new segment labeled by $v$ from the initial vertex to the terminal vertex of the unsaturated segment to obtain a  second generation bounded region.  Otherwise the process of obtaining higher generation bounded regions starting from an unsaturated segment labeled by an $R$-word  passing through $\gamma$ terminates at the previous step.  If the segment labeled by $x_1u$ of $t_i$ is followed by $x_1$, then we can find another unsaturated segment labled by $u$, where we can construct a third generation bounded region. Otherwise, the construction process terminates at the previous step. The construction process terminates after a finite number of sets, because every bounded region constructed through this process shares a boundary edge with the transversal $t_i$, and $t_i$ has finite length.

  \textbf{Fact: 2} \textit{Let $\langle X|u=v\rangle$ be an Adian presentation, such that none of the $R$-word is a subword of the other $R$-word, $u\equiv xsx$, for some $x\in X^+$ and $s\in X^*$, and a suffix of $u$ is a prefix of $v$ (that is, $u\equiv x_1y_1$ and $v\equiv y_1 z_1$, for some $x_1,y_1,z_1\in X^+$), but no prefix of $u$ is a suffix of $v$, then we can't read both the $R$-words starting from some interior vertices of a segment labeled by $u$ and terminating after the terminal vertex of the segment labeled by $u$  along a positively labeled path.}
 
 \begin{proof} \textit{Proof of the fact:} First of all we realize that $x\not\equiv y_1$, otherwise the presentation $\langle X|u=v\rangle$ is not Adian. 
 
 If $y_1$ is a proper suffix of $x$, we read $v$ starting from the segment labeled by $y_1$, and we also read $u$ starting from the segment labeled by $x$ along the same positively labeled path (as shown in the following figure), then there are two possibilities. If the segment labeled by $v$ terminates before the terminal vertex of the segment labeled by $u$, then $v$ is a proper subword of $u$. Which is a contradiction. If the segment labeled by $v$ terminates after the the segment labeled by $u$ (as shown in the following figure), then a longer prefix of $v$ is a suffix of $u$, which contradicts the maximality of $y_1$. 
 
 \begin{figure}[h!]
\centering
\includegraphics[trim = 0mm 0mm 0mm 0mm, clip,width=2.08in]{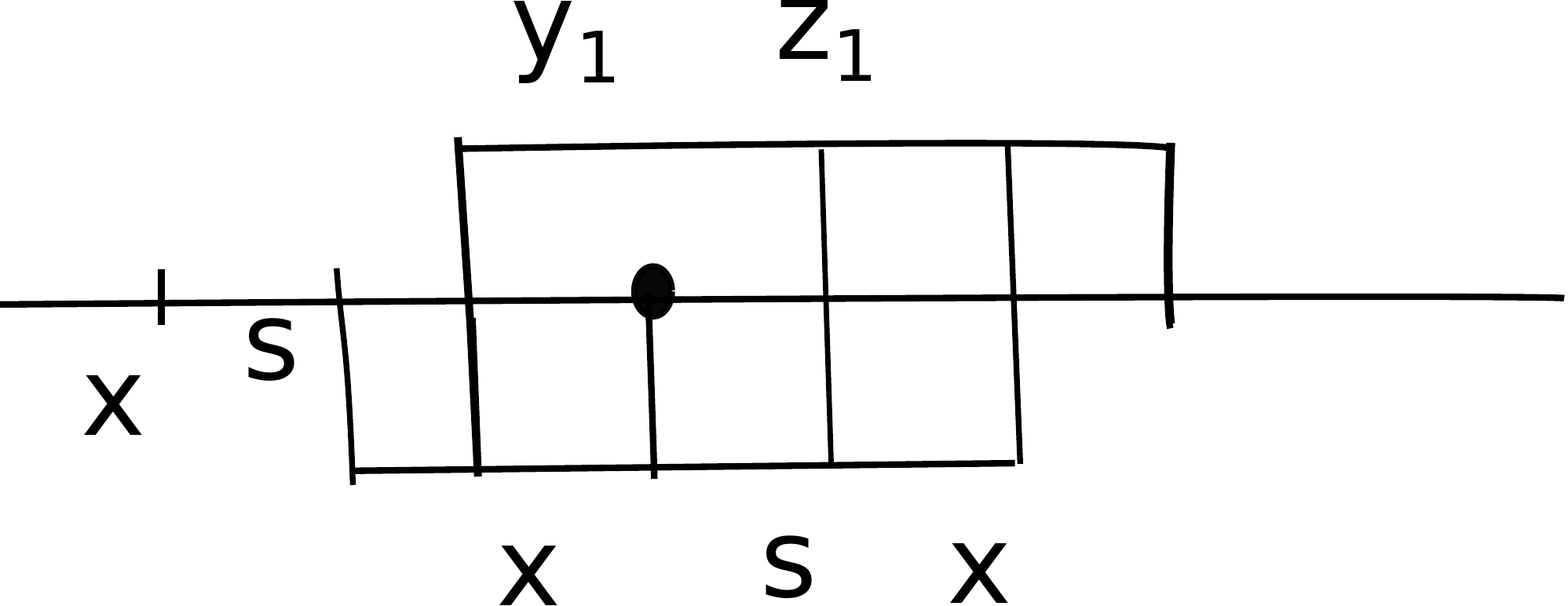}
\caption{ }
\label{fig301}
\end{figure}

 If $x$ is a proper suffix of $y_1$, we read $v$ starting from the segment labeled by $y_1$, and we also read $u$ starting from the segment labeled by $x$ along the same positively labeled path (as shown in the following figure), then once again there are two possibilities. If the segment labeled by $u$ terminates before the terminal vertex of the segment labeled by $v$, then $u$ is a subword of $v$ which is a contradiction. If the terminal vertex of the segment labeled by $u$ lies after the terminal vertex of the segment labeled by $v$, then a suffix of $v$ is a prefix of $u$, which is again a contradiction.  
 
  \begin{figure}[h!]
\centering
\includegraphics[trim = 0mm 0mm 0mm 0mm, clip,width=2.08in]{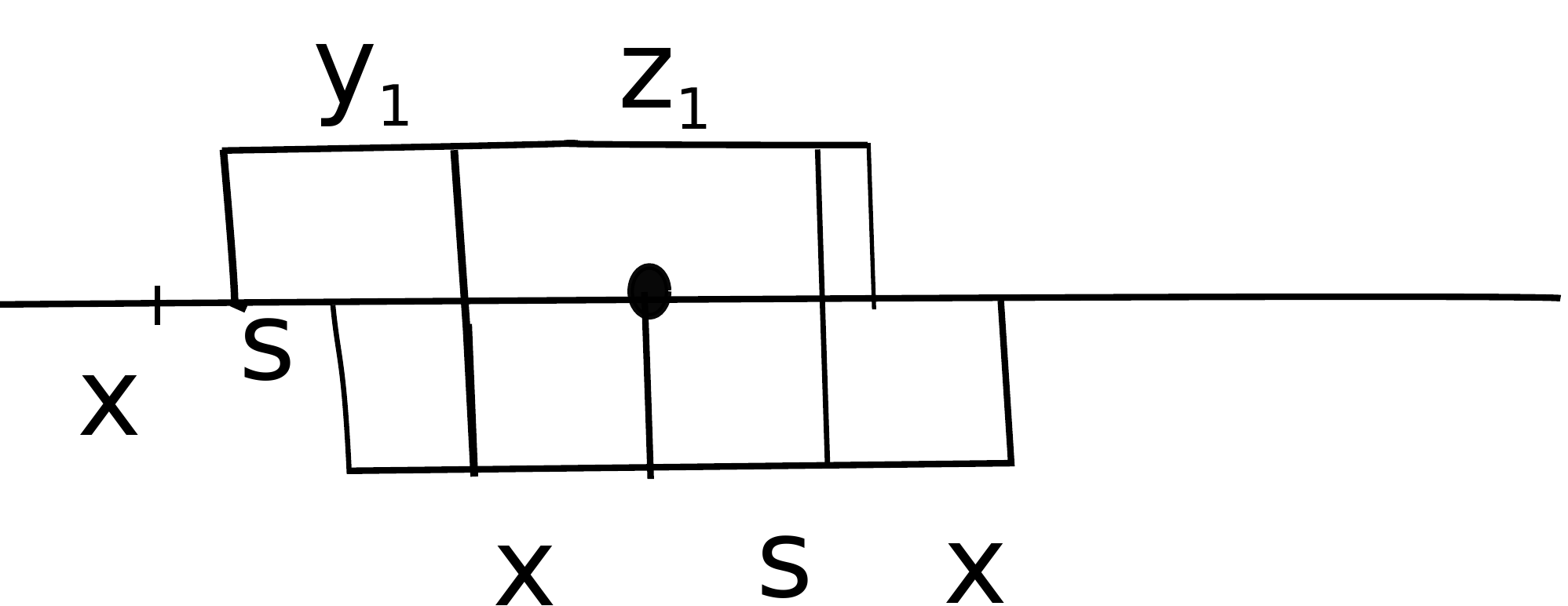}
\caption{ }
\label{fig301}
\end{figure}

  \end{proof}

  If the unsaturated segment of $t_i$ that passes through $\gamma$ is labelled by $v$, then we sew on a segment labeled by $u$ from the initial vertex to the terminal vertex of the segment labeled by $v$ and obtain a first generation bounded region. This can create some more unsaturated segments labeled by an $R$-word. 
  
  On one side of this first generation bounded region, if the segment labeled by $v$ of $t_i$ is followed by a segment labeled by $xs$, then we can find an unsaturated segment labeled by $u$. So, we sew on a segment labeled by $v$ from the initial vertex to the terminal vertex of the unsaturated segment labeled by $u$ and obtain a second generation bounded region. If the segment labeled by $xsv$ of $t_i$ is followed by $x_1$, then we can find an unsaturated segment labeled by $u$. We sew on a segment labeled by $v$ from the initial vertex to the terminal vertex of the unsaturated segment labeled by $u$, and obtain a third generation bounded region. We continue this process and we observe that every higher generation bounded region shares a boundary edge with the transversal $t_i$. Since $t_i$ is a word of a finite length, therefore, this process eventually terminates on this side.  
  
  On the other side of the first generation bounded region, we can't find two unsaturated segments labeled by $u$ and $v$ simultaneously, because of the Fact 2. So, if the segment labeled by $sx$ appears right  after the segment labeled by $v$, then we find an unsaturated segment labeled by $u$, and if the segment labeled by $z_1$ appears right after the segment labeled by $v$, then we can find an unsaturated segment labeled by $v$. If the unsaturated segment is labeled by $u$, then we sew on a segment labeled by $v$ from the initial vertex to the terminal vertex of the unsaturated segment  labeled by $u$ and obtain a second generation bounded region. In this case, the extension process terminates at this step as $v$ doesn't overlap with itself and no suffix of $v$ is a prefix of $u$ either. If the unsaturated segment is labeled by $v$ then we sew on a segment labeled by $u$ from the initial vertex to the terminal vertex of the unsaturated segment labeled by $v$. Once again, we have same situation as before, that is, if the sewing on the segment labeled by $u$ creates an unsaturated segment, then it will be either labeled by $u$ or by $v$, so we repeat the same steps. Note that every bounded region shares a boundary edge with the transversal labeled by $t_i$. Since $t_i$ is a word of finite length therefore, this process will eventually terminate on this side.   
  
  Now we assume that $u\equiv xsx$, $v\equiv yty$, for some $x,y\in X^+$, $s,t\in X^*$. There are two possibilities, either $s\equiv t$ or $s\not\equiv t$.  First, we assume that $s\not\equiv t$. We also assume that a suffix of $u$ is a prefix of $v$ (that is, $u\equiv x_1y_1$, $v\equiv y_1z_1$, for some $x_1,y_1, z_1\in X^+)$) but no prefix of $u$ is a suffix of $v$. In this case $x$ and $y_1$ can't be the same words, otherwise the presentation $\langle X|u=v\rangle$ will not be an Adian presentation. Without loss of generality, we assume that the unsaturated segment passing through the vertex $\gamma$ is labeled by $v$. 
  
  We sew on a segment labeled by $u$ from the initial vertex to the terminal vertex of the segment labeled by $v$. This can create some new unsaturated segments labeled by an $R$-word. On one side of the first generation bounded region, if a segment labeled by $xs$ appears  right before the segment labeled by $v$ in the transversal $t_i$, then we can find an unsaturated segment labeled by $u$. We sew on a segment labeled by $v$ from the initial vertex to the terminal vertex of the unsaturated segment labeled by $u$ and obtain a second generation bounded region.   If the segment labeled by $x_1$ appears right before the segment labeled by $xsv$ in $t_i$, then we find an unsaturated segment labeled by $ u$ that terminates at an interior vertex of the segment labeled by $v$ that was attached in the previous step. If the segment labeled by $yt$ appears right before the segment labeled by $xsv$ in $t_i$ then we find an unsaturated segment labeled by $v$ that terminates at an interior vertex of the segment labeled by $v$ that was attached in the previous step. We can't have both of the unsaturated segments ( one labeled by $u$ and the other one labeled by $v$, both terminating at the interior vertices of the segment labeled by $v$ that was attached in the previous step) together, because of the Fact $2$.  So, in either case, we sew on a segment labeled by the other side of the relation from the initial vertex to the terminal vertex of the unsaturated segment and obtain a third generation bounded region. We observe that every bounded region shares a boundary edge with the transversal $t_i$, but $t_i$ has a finite length, therefore, the extension process on this side of the first generation bounded region eventually terminates after a finite number of steps. Same argument can be used to show that the extension process of the graph terminates after a finite number of steps on the other side of the first generation bounded region. Hence $S\Gamma(t_i)$ remains finite in this case.

 Finally,  we assume that $u\equiv xsx$, $v\equiv yty$, for some $x,y\in X^+$, $s,t\in X^*$ and $s\equiv t$. We also assume that a suffix of $u$ is a prefix of $v$ (that is, $u\equiv x_1y_1$, $v\equiv y_1z_1$, for some $x_1,y_1, z_1\in X^+)$) but no prefix of $u$ is a suffix of $v$. Clearly $x$ and $y_1$ can't be the same words, otherwise the presentation $\langle X|u=v\rangle$ will not be an Adian presentation. 

 We observed that if $u\equiv xsx$, $v\equiv yty$, for some $x,y\in X^+$, $s,t\in X^*$ and $s\equiv t$, then we can have multiple layers of bounded regions. The second layer of bounded regions contain fewer  bounded regions than the first layer. The third layer of bounded regions contains fewer bounded regions than the second layer, and so on. By using the same argument, as in the above paragraphs, we can show that the first layer of bounded regions contain only finite number of bounded regions. Hence $S\Gamma(t_i)$  remains finite in this case.

\textit{Case 4} Without loss of generality, we assume that the unsaturated segment of $t_i$ passing through $\gamma$, that is making $S\Gamma(t_i)$ infinite, is labeled by $v$. We consider the linear automaton of $t_i$ and we sew on a segment labeled by $u$ from the initial vertex to the terminal vertex of the segment labeled by $v$ and obtain a first generation bounded region. We can either find an unsaturated segment labeled by an $R$-word starting from a vertex of $t_{i_{1}}$ and terminating at an interior vertex of the newly attached segment labeled $u$, or an unsaturated segment labeled by an $R$-word starting form an interior vertex of the newly attached segment labeled by $u$ and terminating at a vertex of $t_{i_{2}}$, or we can have both the cases together. We sew on path(s) labeled the other side of the relation from the initial vertex to the terminal vertex of the unsaturated segment(s) labeled by the $R$-word(s) and obtain some second generation  bounded region(s). For each second generation bounded region, we can find an unsaturated segment labeled by an $R$-word that either starts from a vertex that lies before the initial vertex of the segment attached in the previous step and terminates at an interior vertex of that segment, or that starts from an interior vertex of the segment attached in the previous step and terminates after the terminal vertex of that segment or both. So, we sew on a segments labeled by the other sides of the  relation and obtain the third generation bounded regions. If all the bounded regions constructed in this way share a boundary edge with the transversal labeled by $t_i$, then the process of constructing higher generation bounded regions will eventually terminate, because $t_i$ is a word of finite length. So, we can find a smallest positive number $j$, such that at least one of the $(j+1)$-st generation bounded region does not share any boundary edge with the transversal labeled by $t_i$. 

\begin{figure}[h!]
\centering
\includegraphics[trim = 0mm 0mm 0mm 0mm, clip,width=1.8in]{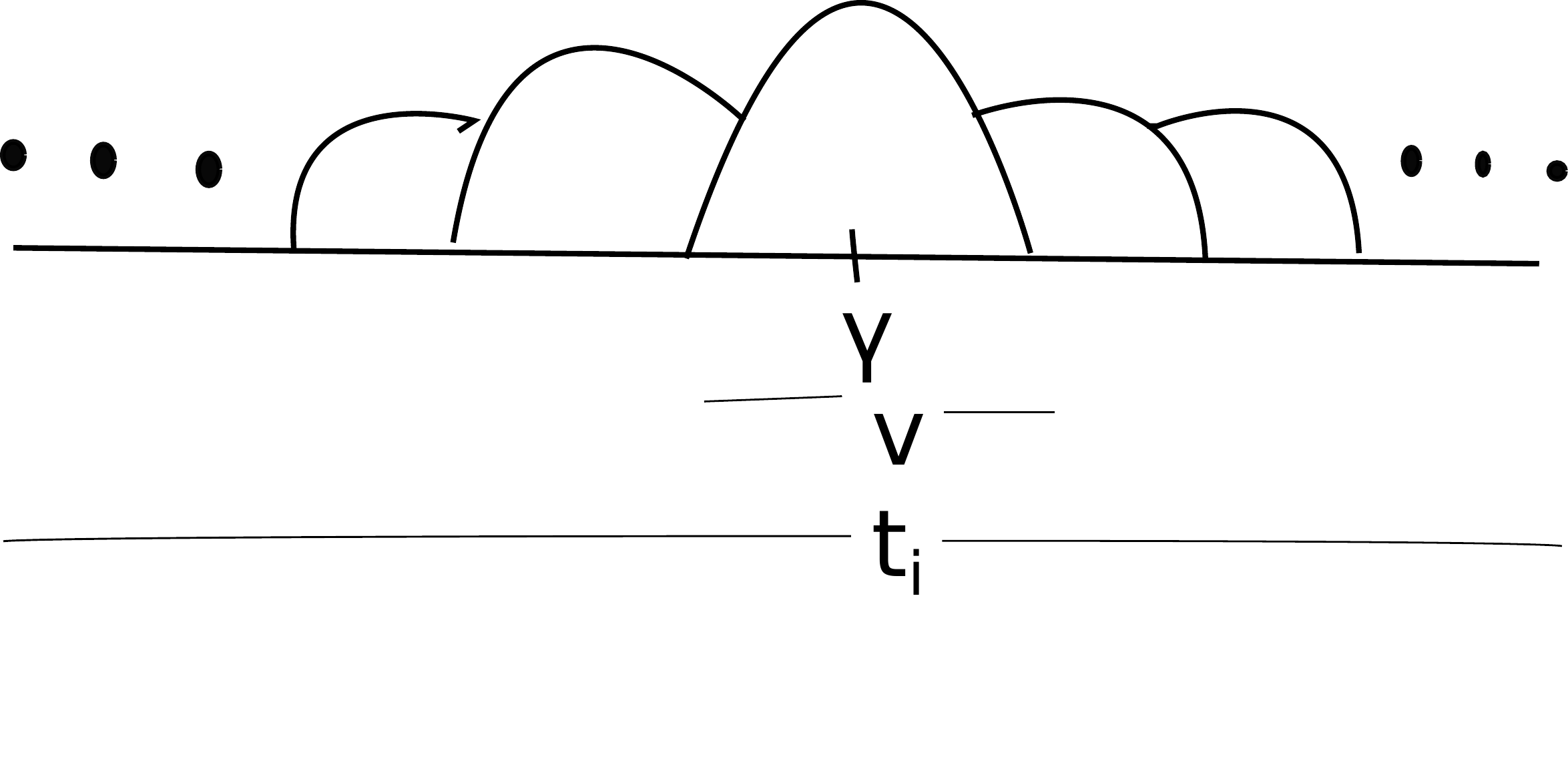}
\caption{ }
\label{fig301}
\end{figure}

We observe that at least one of the vertex between the source and the sink of the $j$-th region lies on the transversal labeled by $t_i$. Without loss of generality, we assume that the source of the source of $j$-th bounded region lies on the transversal labeled by $t_i$. As we know that at least one of the $(j+1)$st bounded region  doesn't share any boundary edge with the transversal $t_i$. The one side of this $(j+1)$st bounded region lies on the positively labeled segment that starts from the source vertex of the $j$-th bounded region, travels along the segment that was attached at the $j$-th step  and terminates at the sink vertex of the $(j-1)$st bounded region. We assume that, this positively labeled segment is labeled by $t$, where $t\equiv ($An $R$-word)(suffix of an $R$-word). 

The segment labeled by $t$ contains exactly two overlapping sub-segments labeled by some $R$-words, one sub-segment is one side of the $j$-th bounded region, and the other sub-segment is one side of the $(j+1)$st bounded region. We observe that the initial vertex of the one side of the $(j+1)$st bounded region lies after the initial vertex of $t$ and the terminal vertex of the one side of the $(j+1)$st bounded regions lies before the terminal vertex of $t$, because the $Star^I(\rho)$ and $Star^O(\rho)$, for some vertex $\rho$ of $S\Gamma(t_i)$ over an Adian presentation $\langle X|u=v\rangle$, can't contain more than two elements.  The out-star set of the initial vertex of $t$ and the in-star set of the terminal vertex of $t$, being the source and sink vertices of some bounded regions, already contain two elements.


We sew on a segment labelled by the other side of the relation from the initial vertex to the terminal vertex of the segment labeled by one side of the $(j+1$)st bounded region, at $t$. According lo our hypothesis, the process of constructing higher and higher generation bounded regions never terminates, therefore, we can find at least one unsaturated segment that either starts from a vertex of $t$ and terminates at an interior vertex of  the newly attached segment, or it starts from an interior vertex of the newly attached segment and terminates at a vertex of $t$. These unsaturated segment(s) label one side of the $(j+2)$nd bounded regions.

If the initial vertex of one side of the $(j+2)$nd bounded region lies before the initial vertex of the one side of the $(j+1)$st bounded region at $t$, then the initial vertex of one side of the $(j+2)$nd bounded region does not coincide with the initial vertex of $t$, because the out-star set of the initial vertex of $t$ already contains two elements.  Similarly, if the terminal vertex of the one side of the $(j+2)$nd bounded region lies after the terminal vertex of  the one side of the $(j+1)$st bounded region at $t$, then the terminal vertex of one side of the $(j+2)$nd bounded region lies before the terminal vertex of $t$, because the in-star set of the terminal vertex of $t$ already contains two elements. We sew on the segment(s) labeled by the other side of the relation from the initial vertex to the terminal vertex of unsaturated segments created in the previous step and obtain some $(j+2)$nd bounded region(s). This step creates some new unsaturated segments for the next generation bounded regions, by our hypothesis. So, it turns out that $S\Gamma(t)$ embeds in $S\Gamma(t_i)$ and $S\Gamma(t)$ is infinite. The initial vertex of $t$ is not a  source vertex of any of $(j+1))$st or higher generation bounded region, so, we can factorize $t$ as $at'$, where $a\in X$. Note that $t'$ is a positive word that contains only one $R$-word as its subword and this $R$-word appear only once in $t'$. also $S\Gamma(t')$ is infinite. This contradicts the hypothesis of the theorem.

We observed that $S\Gamma(t_i)$, for all $i\in I$, are finite. We know that $S\Gamma(t_i)$ embeds in $S\Gamma(w)$,  and no two edges get identified with each other in the construction of $S\Gamma(w)$, over the Adian presentation $\langle X|u=v\rangle$, by Lemma \ref{1}. We start the expansion of $S_0$ from the unsaturated segments, that are labeled by the $R$-words and pass through the vertex $\gamma$.  We sew on segments labeled by the other side of the relation from the initial vertex to the terminal vertex of the unsaturated segments labeled by  the $R$-words, This step may create some new unsaturated  segments labeled by the $R$-words. So, we continue the process of expansion of $S_0$ along the transversals labeled by $t_i$'s. Basically, we have been constructing the Sch\"{u}tzenberger graphs of $t_i$'s along the corresponding transversals of $S_0$.  Since $S\Gamma(t_i)$ are finite for all $i\in I$, therefore, the expansion of $S_0$ eventually terminates. Moreover, the star set of a vertex of $S\Gamma(w)$, over the Adian presentation $\langle X|u=v\rangle$, can't contain more than four elements and the star sets of some of the vertices of $S_0$ (a subgraph of $S\Gamma(w)$)  is increasing as a consequence of this expansion process. So,  we'll obtain a finite graph $S_1$ that will have no unsaturated  segment labeled by the $R$-words. This finite graph is a Sch\'{u}tzenberger graph of the word $w$.

\end{proof}

\section{Applications}

\begin{proposition} Let $\langle X|u=v\rangle$ be an Adian presentation such that $u\equiv xsx$ or/and $v\equiv yty$ for some $x,y\in X^+$ and $s,t\in X^*$, where $x$ and $y$ are the maximal prefixes of $u$ and $v$ that are also the suffixes of $u$ and $v$, and there is no overlap between $u$ and $v$. Then $S\Gamma(w)$, for all $w\in X^+$, is finite. 
 
\end{proposition}

\begin{proof}  We just need to show that the Sch\"{u}tzenberger graph of every positive word that contains an $R$-word only once, is finite. First, we assume that $u\equiv xsx$ for some $x\in X^+$ and $s\in X^*$ (i.e., $u$ has same prefix and suffix), whereas $v$ is any positive word such that no proper prefix of $v$ is a suffix of $v$ and $v$ does not overlap with $u$ either. Let $w$ be any positive word that contains an $R$-word only once as it's subword. We show that the Sch\"{u}tzenberger graph of $w$ is finite. We consider the linear automaton $(\alpha_0,\Gamma_0(w),\beta_0)$.
 
 If $w$ contains $u$ as it's subword then we sew on a path labeled by $v$, from the initial vertex to the terminal vertex of the segment labeled by $u$  of  $(\alpha_0,\Gamma_0(w),\beta_0)$, and obtain $(\alpha_1,\Gamma_1(w),\beta_1)$. Since $v$ does not overlap with itself or with $u$ , therefore we can't find an unsaturated segment labeled by an $R$-word, where we can sew on a new path labeled by an $R$-word. So, $S\Gamma(w)$ remains finite.

 If $w$ contains $v$ as it's subword then we sew on a path labeled by $u$ from the initial vertex to the terminal vertex of the segment labeled by $v$ of  $(\alpha_0,\Gamma_0(w),\beta_0)$ to obtain  $(\alpha_1,\Gamma_1(w),\beta_1)$. Since $u$ overlaps with itself, therefore if the segment labeled by $v$ is followed by a segment labeled by $xs$ or a segment labeled by $sx$ is followed by $v$, in the automaton  $(\alpha_1,\Gamma_1(w),\beta_1)$, then we can find an unsaturated segment labeled by $u$ in $(\alpha_1,\Gamma_1(w),\beta_1)$  where we can attach a new segment labeled by $v$ from the initial vertex to the terminal vertex of the unsaturated segment(s) to obtain $(\alpha_2,\Gamma_2(w),\beta_2)$.  Otherwise the underlying graph of$(\alpha_1,\Gamma_1(w),\beta_1)$ is $S\Gamma(w)$.  The $R$-word $v$ does't overlap with itself or with $u$, therefore we can't find an unsaturated segment labeled by an $R$-word in $(\alpha_2,\Gamma_2(w),\beta_2)$ where we can attach a new path labeled by the other $R$-word to obtain $(\alpha_3,\Gamma_3(w),\beta_3)$. So, $S\Gamma(w)$ is finite in this case.

 If $u\equiv xsx$ and $v\equiv yty$ for some $x,y\in X^+$ and $s,t\in X^*$, while $u$ and $v$ do not overlap with each other. Note that $x$ and $y$ cannot be the same words otherwise the presentation $\langle X|u=v\rangle$ will not be an Adian presentation. However, $s$ and $t$ can be the same words. We show that $S\Gamma(w)$ remains finite in both the cases, when $s\equiv t$ and when $s\not\equiv t$.

 We first assume that $s\not\equiv t$. The word $w$ contains an $R$-word as it's subword, so without loss of generality we assume that $w$ contains $v$ as it's subword and show that $S\Gamma(w)$ is finite. A similar argument can be used to show that $S\Gamma(w)$ is finite when $u$ is the subword of $w$.  We construct the linear automaton of $w$, $(\alpha_0,\Gamma_0(w),\beta_0)$. We sew on a segment labeled by $u$ from the initial vertex to the terminal vertex of the segment labeled by $v$ and obtain $(\alpha_1,\Gamma_1(w),\beta_1)$.  If the word $v$ is followed by the word $xs$ or the word $sx$ is followed by $v$ in $w$, then we can find at most two unsaturated segments labeled by$u$ in $(\alpha_1,\Gamma_1(w),\beta_1)$. So,  we sew on a segment labeled by $v$ from the initial vertex to the terminal vertex of the unsaturated segments labeled by $u$ and obtain $(\alpha_2,\Gamma_2(w),\beta_2)$. If the word $xsv$ is followed by the word $yt$ or the the word $ty$ is followed by $vsx$, then we can find some unsaturated segments labeled by $v$ in $(\alpha_2,\Gamma_2(w),\beta_2)$. So, once again we apply full $P$-expansion to obtain $(\alpha_3,\Gamma_3(w),\beta_3)$. We observe that every time we apply the  full $P$-expansion, we utilize  some of the segments of the linear automaton of $w$.  The process of application of elementary expansion will eventually  terminate, as $w$ is a word of finite length. So, $S\Gamma(w)$ is finite in this case. 

Now we assume that $s\equiv t$ and $w$ contains $v$ as it's subword and show that $S\Gamma(w)$ is finite. We consider the liner automaton of $w$, $(\alpha_0,\Gamma_0(w),\beta_0)$. We sew on a segment labeled by $u$ from the initial vertex to the terminal vertex of the segment labeled by $v$ and obtain $(\alpha_1,\Gamma_1(w),\beta_1)$. If the word $v$ is followed by the word $xs$ or the word $sx$ is followed by $v$ in $w$, then we can find at most two unsaturated segments labeled by $u$ in $(\alpha_1,\Gamma_1(w),\beta_1)$. So,  we sew on a segment labeled by $v$ from the initial vertex to the terminal vertex of the unsaturated segments labeled by $u$ and obtain $(\alpha_2,\Gamma_2(w),\beta_2)$. if the word $xsvsx$ is a subword of $w$, then we sew on the segments labeled by $v$ on both sides of the bounded region (created in the first step) and this creates an unsaturated segment labeled by $v$ starting from an interior vertex of the segment labeled by $v$ (that was just attached) and terminating the other segment labeled by $v$ (that was just attached). Now, if the word $xsv$ is followed by $ys$, or the word $sy$ is followed by $vsx$ in $w$, or the word $ysxsvsxsy$ is a subword of $w$ then we apply full $P$-expansion to obtain  $(\alpha_3,\Gamma_3(w),\beta_3)$. We continuously apply the procedure of full $P$-expansion to obtain the higher generation approximate graphs and observe that the Sch\"{u}tzenberger graph of $w$ grows in a pyramid shape. Since $w$ is a word of finite length, therefore, the base of this pyramid stops growing after a finite number of steps. 
We also observe that if the base of the pyramid contain $N$ number of bounded regions then the second layer contains $(N-2)$ bounded regions, the third layer contains $(N-4)$ bounded regions and son and so forth. Hence $S\Gamma(w)$ is finite in this case.

\end{proof}

\begin{proposition} 
Let $\langle X|u=v\rangle$ be an Adian presentation such that $u\equiv xy$ and $v\equiv yz$ or $v\equiv zx$, for some $x,y,z\in X^+$. (If $v\equiv yz$ then $y$ is the maximal suffix of $u$ that is a prefix of $v$ and no prefix of $u$ is a suffix of $v$. If $v\equiv zx$ then $x$ is the maximal prefix of $u$ that is a suffix of $v$ and no suffix of $u$ is a prefix of $v$). Then $S\Gamma(w)$, for all $w\in X^+$, is finite. 

\end{proposition}

\begin{proof} We show that the Sch\"{u}tzwnberger graph of every positive word that contains an $R$-word only once, is finite. Without loss of generality, we assume that $u\equiv xy$ and $v\equiv yz$, where $y$ is the maximal suffix of $u$ that is a prefix of $v$, no prefix of $u$ is a suffix of $v$. There are two cases to consider. 

\begin{enumerate}

\item None of the $R$-word has same prefix and suffix.

\item One or both of the $R$-words have same prefixes and suffixes. 

\end{enumerate}

Let $w\in X^+$ be a positive word that contains an $R$-word as it's subword only once. Without loss of generality we assume that $w$ contains $v$ as it's subword and show that $S\Gamma(w)$ is finite. A similar argument can be used to show that $S\Gamma(w)$ is finite, when $w$ contains $u$ as it's subword. 

\textit{Case 1:} We assume that none of the $R$-words have same prefix and suffix. We construct the linear automaton of $w$, $(\alpha_0,\Gamma_0(w),\beta_0)$.. We sew on a segment labeled by $u$ from the initial vertex to the terminal vertex of the segment labeled by $v$ and obtain $(\alpha_1,\Gamma_1(w),\beta_1)$. If the word $z$ is followed by $v$ in $w$, then we can find an unsaturated segment labeled by $v$ in $(\alpha_1,\Gamma_1(w),\beta_1)$. we apply full $P$-expansion and obtain $(\alpha_2,\Gamma_2(w),\beta_2)$.  If the word $z$ is followed by the word $vz$ in $w$, then we repeat the process of full $P$-expansion and obtain $(\alpha_3,\Gamma_3(w),\beta_3)$. We observe that the Sch\"{u}tzenberger graph of $w$ is only growing in one direction and at every iterative step of full $P$-expansion, it is utilizing  a segment of the transversal labeled by $w$. So, the sequence of approximate graphs of $w$ terminates after a finite number of steps, as $w$ is a word of finite length. 

\textit{Case 2:} First, we assume that $u\equiv x_1sx_1$ for some $x_1\in X^+$ and $s\in X^*$ (i.e., $u$ has same prefix and suffix), whereas $v$ is any positive word such that no proper prefix of $v$ is a suffix of $v$ and $v$ overlaps with $u$. 

We construct the linear automaton of $w$, $(\alpha_0,\Gamma_0(w),\beta_0)$. If $w$ contains $u$ as as it's subword, then we sew on a path labeled by $v$ from the initial vertex to the terminal vertex of the segment labeled by $u$ to obtain $(\alpha_1,\Gamma_1(w),\beta_1)$. If $u$ is followed by $x$ in $w$, then we can find an unsaturated segment labeled by $u$ in $(\alpha_1,\Gamma_1(w),\beta_1)$. We sew on a segment labeled by $v$ from the initial vertex to the terminal vertex of this unsaturated segment and obtain $(\alpha_2,\Gamma_2(w),\beta_2)$. If $xu$ is followed by $x$ in $w$ then we can find an unsaturated segment labeled by $u$ in  $(\alpha_2,\Gamma_2(w),\beta_2)$ and we apply Full $P$-expansion to obtain  $(\alpha_3,\Gamma_3(w),\beta_3)$.  We continue the process of Full $P$-expansion and observe that every bounded region shares a boundary segment with the transversal labeled by $w$. Since $w$ is a word of finite length, therefore the process of Full $P$-expansion eventually terminates. Hence $S\Gamma(w)$ remains finite in this case. 
 
 If $w$ contains $v$ as it's subword, then we sew on a segment labeled by $u$ from the initial vertex to the terminal vertex of the segment labeled by $v$ to obtain  $(\alpha_1,\Gamma_1(w),\beta_1)$.    The approximate graph $(\alpha_1,\Gamma_1(w),\beta_1)$ contains only one bounded region. On one side of this first generation bounded region, if the segment labeled by $v$ of $w$ is followed by a segment labeled by $x_1s$, then we can find an unsaturated segment labeled by $u$. So, we sew on a segment labeled by $v$ from the initial vertex to the terminal vertex of the unsaturated segment labeled by $u$ and obtain a second generation bounded region. If the segment labeled by $x_1sv$ of $w$ is followed by $x$, then we can find an unsaturated segment labeled by $u$. We sew on a segment labeled by $v$ from the initial vertex to the terminal vertex of the unsaturated segment labeled by $u$, and obtain a third generation bounded region. We continue this process and we observe that every higher generation bounded region shares a boundary edge with the transversal labeled by $w$. Since $w$ is a word of a finite length, therefore, this process eventually terminates on this side. 
  
  On the other side of the first generation bounded region, we can't find two unsaturated segments labeled by $u$ and $v$ simultaneously, because of the Fact 2. So, if the segment labeled by $sx_1$ appears right  after the segment labeled by $v$ in the transversal labeled by $w$, then we find an unsaturated segment labeled by $u$, and if the segment labeled by $z$ appears right after the segment labeled by $v$, then we can find an unsaturated segment labeled by $v$. If the unsaturated segment is labeled by $u$, then we sew on a segment labeled by $v$ from the initial vertex to the terminal vertex of the unsaturated segment  labeled by $u$ and obtain a second generation bounded region. In this case, the extension process terminates at this step as $v$ doesn't overlap with itself and no suffix of $v$ is a prefix of $u$ either. If the unsaturated segment is labeled by $v$ then we sew on a segment labeled by $u$ from the initial vertex to the terminal vertex of the unsaturated segment labeled by $v$. Once again, we have same situation as before, that is, if the sewing on the segment labeled by $u$ creates an unsaturated segment, then it will be either labeled by $u$ or by $v$, so we repeat the same steps. Note that every bounded region shares a boundary edge with the transversal labeled by $t_i$. Since $t_i$ is a word of finite length therefore, this process will eventually terminate on this side.

 Now we assume that $u\equiv x_1sx_1$, $v\equiv y_1ty_1$, for some $x_1,y_1\in X^+$, $s,t\in X^*$ and $s\not\equiv t$. There are two possibilities, either $s\equiv t$ or $s\not\equiv t$.  First, we assume that $s\not\equiv t$. In this case $x_1$ and $y$ can't be the same words, otherwise the presentation $\langle X|u=v\rangle$ will not be an Adian presentation. Without loss of generality, we assume that $w$ contains $v$ as it's subword. 
  
  We sew on a segment labeled by $u$ from the initial vertex to the terminal vertex of the segment labeled by $v$  to obtain  $(\alpha_1,\Gamma_1(w),\beta_1)$. This can create some new unsaturated segments labeled by an $R$-word in $(\alpha_1,\Gamma_1(w),\beta_1)$. On one side of the first generation bounded region, if a segment labeled by $x_1s$ appears  right before the segment labeled by $v$ in the transversal labeled by $w$, then we can find an unsaturated segment labeled by $u$. We sew on a segment labeled by $v$ from the initial vertex to the terminal vertex of the unsaturated segment labeled by $u$ and obtain a second generation bounded region.   If the segment labeled by $x$ appears right before the segment labeled by $x_1sv$ in the transversal labeled by $w$, then we find an unsaturated segment labeled by $ u$ that terminates at an interior vertex of the segment labeled by $v$ that was attached in the previous step. If the segment labeled by $y_1t$ appears right before the segment labeled by $x_1sv$ in the transversal labeled by $w$ then we find an unsaturated segment labeled by $v$ that terminates at an interior vertex of the segment labeled by $v$ that was attached in the previous step. We can't have both of the unsaturated segments ( one labeled by $u$ and the other one labeled by $v$, both terminating at the interior vertices of the segment labeled by $v$ that was attached in the previous step) together, because of the Fact $2$.  So, in either case, we sew on a segment labeled by the other side of the relation from the initial vertex to the terminal vertex of the unsaturated segment and obtain a third generation bounded region. We observe that every bounded region shares a boundary edge with the transversal labeled by $w$, but $w$ has a finite length, therefore, the extension process on this side of the first generation bounded region eventually terminates after a finite number of steps. Same argument can be used to show that the extension process of the graph terminates after a finite number of steps on the other side of the first generation bounded region. Hence $S\Gamma(w)$ is finite in this case.

 Finally,  we assume that $u\equiv x_1sx_1$, $v\equiv y_1ty_1$, for some $x_1,y_1\in X^+$, $s,t\in X^*$ and $s\equiv t$. Clearly $x_1$ and $y$ can't be the same words, otherwise the presentation $\langle X|u=v\rangle$ will not be an Adian presentation. 

 We observed that if $u\equiv x_1sx_1$, $v\equiv y_1ty_1$, for some $x_1,y_1\in X^+$, $s,t\in X^*$ and $s\equiv t$, then we can have multiple layers of bounded regions. The second layer of bounded regions contain fewer  bounded regions than the first layer. The third layer of bounded regions contains fewer bounded regions than the second layer, and so on. By using the same argument, as in the above paragraphs, we can show that the first layer of bounded regions contains only finite number of bounded regions. Hence $S\Gamma(w)$  is finite in this case.

\end{proof}

\textbf{Acknowledgement} The author of this paper is thankful to Robert Ruyle for his useful suggestions.

\bibliographystyle{plain}

\end{document}